\documentclass{article}
\usepackage{latexsym, amssymb, enumerate, amsmath, amsthm,pdfsync, multicol}
\usepackage{fullpage}
\usepackage{cleveref}
\usepackage{cite}
\usepackage{url}
\usepackage{esint}
\usepackage{dsfont}
\usepackage{color}

\crefname{assumption}{assumption}{assumptions}
\crefname{thm}{theorem}{theorems}
\crefname{lem}{lemma}{lemmas}
\crefname{cor}{corollary}{corollaries}
\crefname{prop}{proposition}{propositions}
\Crefname{theorem}{Theorem}{Theorems}

\newcommand{\E}{\mathbb{E}}

\newcommand{\PP}{\mathbb{P}}
\newcommand{\R}{\mathbb{R}}
\newcommand{\T}{\mathbb{T}}
\newcommand{\Z}{\mathbb{Z}}

\newcommand{\1}{\mathds{1}}
\let\phi\varphi

\newtheorem{thm}{Theorem}[section]
\newtheorem{prop}[thm]{Proposition}
\newtheorem{lem}[thm]{Lemma}
\newtheorem{cor}[thm]{Corollary}

\newtheorem{rem}[thm]{Remark}

\begin{document}

\title{Propagation of solutions to the Fisher-KPP equation with slowly decaying initial data}
\author{Christopher Henderson
\thanks {Department of Mathematics, Stanford University, Stanford, CA 94305, (chris@math.stanford.edu).}}
\maketitle

\begin{abstract}
The Fisher-KPP equation is a model for population dynamics that has generated a huge amount of interest since its introduction in 1937.  The speed with which a population spreads has been computed quite precisely when the initial data decays exponentially.  More recently, though, the case when the initial data decays more slowly has been studied.  Building on the results of Hamel and Roques '10, in this paper we improve their precision for a broader class of initial data and for a broader class of equations.  In particular, our approach yields the explicit highest order term in the location of the level sets.  In addition, we characterize the profile of the inhomogeneous problem for large times.
\end{abstract}

\section{Introduction}

In this article, we study the long time behavior of solutions to the Fisher-KPP equation
\begin{equation}\begin{cases}\label{e_kpp}
u_t = u_{xx} + f(x,u), ~~~t > 0, x\in \R,\\
u(0,x) = u_0(x), ~~~x\in \R,
\end{cases}\end{equation}
where $f$ is an $L$-periodic function in $x$, for some $L>0$, which
vanishes when $u$ is $0$ and $1$ and where $u_0$ is a function taking
values between 0 and 1 that decays slowly as $x$ tends to infinity.
This equation comes from adding a diffusive term to the logistic
equation, and it was originally introduced, with $f(x,u) = u(1-u)$, as
model for population dynamics in the first half of the twentieth
century~\cite{Fisher,KPP}.  In this interpretation, $u$ stands for the
local population density normalized by the carrying capacity, while
$f$ accounts for the effect of an inhomogeneous environment on the
growth of this population up to the carrying capacity.

Fisher and Kolmogorov, Petrovskii, and Piskunov showed that if the
initial data is localized then the population would spread at a
constant rate in the form of a traveling wave.  This has been made
more precise and succinct over time~\cite{Uchiyama, Lau, Bramson78,
  Bramson83, Gartner, HNRR12, HNRR13, HendersonVS, Roberts}.  In these works, the
highest order terms in the spreading rate are obtained by studying the
linearized problem.  When $u_0$ is not
localized, but decays exponentially, like $e^{-\lambda x}$, the behavior is 
similar.  Here,
the population converges to a traveling wave, the speed of which is determined by the order of the exponential decay and is
given again by the linearized problem~\cite{Uchiyama}.

Recently, the case when $u_0$ decays slower than any exponential and
$f$ is homogeneous was investigated by Hamel and
Roques~\cite{HamelRoques}.  They show that $u$ becomes asymptotically
flat in time, which we will state precisely later.   This implies that a traveling wave cannot give the limiting
behavior of $u$.  As
a result, they investigate the location of the level sets of $u$ and
show that, for each $m\in (0,1)$, the level sets satisfy
\begin{equation}\label{e_const}
	\{x\in \R: u(t,x) = m\} \subset \left[ u_0^{-1}\left(C_m e^{-f'(0)t}\right),u_0^{-1}\left(c_m e^{-f'(0)t}\right)\right],
\end{equation}
where $c_m < C_m$ are constants depending on $m$.  These
results hold for only a small class of initial data; for more general
data, they show that level set is bounded above and below by $u_0^{-1}\left( \exp -(f'(0)\pm \epsilon)t\right)$ for any $\epsilon>0$.  The main results of this paper are twofold:  firstly, in the homogeneous setting, we explicitly characterize the constant in~\eqref{e_const} for a broad class of the admissible initial data; secondly, we obtain equally precise results in the previously unstudied inhomogeneous setting.  In fact, in the inhomogeneous setting, we show that $u$ grows at each point according to the solution to a generalized logistic equation.  This provides an analogue to the asymptotic flatness result of Hamel and Roques in the inhomogeneous setting.

\section{Main results}\label{s_results}

We now state our results more precisely.  First we assume that there
is a $\delta>0$ such that $f(x,u)$ is a positive, $C^{1,\delta}$ function with respect to $u$ that satisfies
\begin{equation}\label{e_f}
\begin{cases}
	f(x,u) \leq f_u(x,0) u \text{ for all } u \in [0,1],\\
	\frac{f(x,u)}{u} \text{ is decreasing in } u,\text{ and}\\
	f_u(x,0) > 0 \text{ and}\\
	f_u(x,1) < 0
\end{cases}
\end{equation}
We note that some of the assumptions on the signs of $f$ and its
derivatives can be relaxed, but we use them here to simplify the
arguments.

We assume that $u_0$ is a positive function that decays to 
zero slower than any exponential.  
By this we mean that, for any $\epsilon \in (0,\infty)$
\begin{equation}\label{e_superexponential}
	\lim_{x\to\infty} u_0(x) = 0,
	~~\text{ and } ~~
	\lim_{x\to\infty} u_0(x) e^{\epsilon x} = \infty.
\end{equation}
For instance, functions which decay algebraically as $x$ tends to infinity satisfy the conditions in~\eqref{e_superexponential}.
Though our technique
can handle any functions of this type, the presentation is much
simpler if we restrict to initial data, $u_0$, which are eventually
monotonic and which decay slower than $e^{-\beta \sqrt{x}}$ for any $\beta$.  Namely, we assume that there is a point $x_0>0$ such that
\begin{align}\label{e_monotonic}
	&u_0'(x) < 0 \text{ for all } x\geq x_0,
	~~\text{ and }\\
	\label{e_regular}
	&\lim_{x\to\infty} \frac{u_0'(x)}{u_0(x)/\log(u_0(x))} = 0.
\end{align}
We will comment on the effect of this assumption later.  In addition, since we are studying the effect of a slowly decaying tail at $\infty$, we additionally assume that $u_0$ is front-like.  In other words, we assume that
\begin{equation}\label{e_left_limit}
	\liminf_{x\to-\infty} u_0(x) > 0.
\end{equation}
Again, we note that the positivity assumption on $u_0$ can be relaxed, but we use it here to simplify the presentation.

Classical results address the case when $u_0 \sim e^{-\lambda
  x}$~\cite{Fisher,KPP,Bramson78,Bramson83, Uchiyama,Lau}.  If
$\lambda > 1$ then $u$ converges to a traveling wave with speed $2$ in
the correct moving frame, while if $\lambda < 1$, $u$ converges to a
traveling wave with speed $\lambda + 1/\lambda$ in the moving frame.
In short, this means that the level sets of $u_0$ move with a constant
speed, which increases to infinity as $\lambda$ decreases to zero, and
that any two level sets travel at the same speed and remain a bounded
distance from one another for all time.  Hamel and Roques point out in
\cite{HamelRoques} that, using these as a family of sub-solutions, we
can see immediately that the level sets of any solution with initial
data like \cref{e_superexponential} will spread to the right super
linearly in time.  We now calculate the speed of these level sets
explicitly.

\subsection{The homogeneous case}

Our first result concerns the spreading of the level sets in the
homogeneous setting.
\begin{thm}\label{p_homogeneous}
  Suppose that $u$ solves \cref{e_kpp} with $f$ depending only on $u$
  and not on the space variable, $x$.
  Suppose further that $f$ satisfies condition \eqref{e_f} and that $u_0$ satisfies
  conditions~\eqref{e_superexponential} and \eqref{e_left_limit}.  In
  addition, let $u_0$ satisfy conditions \eqref{e_monotonic} and \eqref{e_regular}.  Let $\phi$ be the
  unique (up to translation) global in time solution to
\[
	\phi_t = f(\phi).
\]
Then, for each $m \in (0,1)$, define $T_m \in \R$ to be the unique
time so that $\phi(T_m) = m$.  Then, for any $r>0$, we have
\begin{equation}\label{e_limsup}
	\limsup_{T\to\infty} \sup_{x \geq u_0^{-1}\left(\phi(T_m - T)\right) - rT} u(T,x) \leq m.
\end{equation}
and
\begin{equation}\label{e_liminf}
	\liminf_{T\to\infty} \inf_{x \leq u_0^{-1}\left(\phi(T_m - T)\right) + rT}  u(T,x) \geq m.
\end{equation}
\end{thm}
In order to make this result clearer, we first consider a specific example.  If $f$ is the standard Fisher-KPP non-linearity, i.e.~$f(u) = u(1-u)$, then
\[
	\phi(T) = \frac{e^{T} }{1 + e^T}, ~~~~\text{ implying that }~~~ \phi(T_m - T) = \frac{m}{1-m} e^{-T} + O(e^{-2T}),
\]
where $O(e^{-2T})$ is bounded by a constant multiple of $e^{-2T}$ for
all $T$.  For a general non-linearity, $f$, one can easily see (by
linearization) that, up to leading order, $\phi(T_m - T) = c_m e^{-f'(0)T}$, where $c_m$ is a constant depending only on $f$ and $m$.  In addition, defining the level set of $u$ as
\[
	E_m(t) \stackrel{\text{def}}{=} \{x: u(t,x) = m\},
\]
if we look at, e.g., initial data such that $u_0(x) = x^{-\alpha}$ for large enough $x$, \cref{p_homogeneous} implies that
\begin{equation}\label{e_kpp_speed}
	E_m(t) \sim \left( \frac{1-m}{m} e^{t}\right)^{1/\alpha}.
\end{equation}
Notice that the non-linearity affects the constant $\frac{1-m}{m}$, as
the linearized equation would yield only $1/m$.  The fact that the
full non-linearity affects the highest order behavior is in
contrast to the case when $u_0$ decays exponentially, where the
linearized equation completely determines the highest order behavior.

An easy corollary of \cref{p_homogeneous} is the following.
\begin{cor}\label{p_level_sets}
For any $\epsilon>0$, under the assumptions of \cref{p_homogeneous}, there exists $T_0>0$ depending on all parameters such that if $T > T_0$ then
\[
	E_m(T) \subset \left[ u_0^{-1}\left(\phi(T_{m+\epsilon}-T)\right), u_0^{-1}\left(\phi(T_{m-\epsilon} - T)\right)\right].
\]
\end{cor}
Again, it is instructive to look at the example that we considered
above.  
\Cref{p_level_sets} implies that, for any $\epsilon$ and any $t$ large enough, we
have
\[
	E_m(t) \subset \left[ \left(\frac{m}{1-m} - \epsilon\right)^{1/\alpha} e^{t/\alpha}, \left(\frac{m}{1-m}+\epsilon\right)^{1/\alpha} e^{t/\alpha}\right].
\]

\begin{rem}\label{r_differentiability}
We comment briefly on the assumption on the decay of the first derivative of $u_0$, \cref{e_regular}.  In practice, this means that $u_0$ decays slower than $e^{-\sqrt{x}}$.  We could remove this restriction, assuming only that $u_0$ satisfies \cref{e_superexponential}.  In this case, we would obtain a linear in time error term in \cref{e_limsup,e_liminf} and in \cref{p_level_sets}.  In other words, there is a constant $r>0$ such that, for any $m\in (0,1)$ and $\epsilon >0$, we have
\[
	E_m(T) \subset \left[ u_0^{-1}\left( \phi(T_{m+\epsilon} - T)\right) - rT, u_0^{-1}\left(\phi(T_{m-\epsilon} - T)\right) + rT \right],
\]
for $T$ large enough.  We omit proving this here, in the interest of presenting a single, clear argument and result.
\end{rem}

\subsection{The inhomogeneous case}

When the non-linearity, $f$, depends periodically on $x$, the level
sets fail to capture the spreading behavior.  Intuitively this is
because the function $u_0$ is very flat at infinity, so it will grow
as if it were a solution to the periodic, inhomogeneous Fisher-KPP
equation.  In particular, $u$ will oscillate in space enough that the width of the level sets will grow to infinity as $t$ tends to infinity.  In other words, the distance between the rightmost and leftmost points of the level set $E_m(t)$ tends to infinity.

To circumvent this, we look instead at the locations of the intervals of length $L$, the length of the period of $f$,
of a given mean oscillation.  In other words, we seek to describe the
location of the average-level
\[
	\overline{E}_m(t) \stackrel{\text{def}}{=} \left\{x\in \R: \fint_x^{x+L} u(t,y) dy = m\right\},
\]
where we define, for any interval $[a,b]$ and any function $\gamma(x)$,
\[
	\fint_a^b \gamma(x) dx \stackrel{\text{def}}{=} \frac{1}{|b-a|} \int_a^b \gamma(x) dx.
\]
Using $\overline{E}_m(t)$, we find similar results to \cref{p_homogeneous} in the
inhomogeneous setting.

As we alluded to above, instead of the logistic equation, the behavior
of $u$ on each interval of length $L$ is captured by the global in time solution
to
\begin{equation}\label{e_logistic_global}
	\begin{cases}
		\varphi_t = \varphi_{xx} + f(x,\varphi), ~~\text{ for all } (t,x)\in\R\times \T,\\
		\fint \varphi(0,x)dx = \frac{1}{2},
	\end{cases}
\end{equation}
where $\T = [0,L]$ with the ends identified.  In other words, $\phi$ is an $L$-periodic solution to the equation above.  We briefly note that, although
the techniques to show the well-posedness of \cref{e_logistic_global}
are similar to those used in the theory of well-posedness for the
Fisher-KPP equation, see e.g. \cite{BerestyckiNirenberg}, we are
unable to find a treatment of it in the literature.  As such, we
dedicate \cref{s_wellposed} to proving the existence and uniqueness of the solution to this problem.


We define one last quantity before we state our results.  For each $m \in (0,1)$ and each $T>0$, we may solve the terminal value problem
\begin{equation}\label{e_inhomogeneous_logistic_tvp}
\begin{cases}
	\phi^{T,m}_t = \phi^{T,m}_{xx} + f(x, \phi^{T,m}), ~~\text{ on } [0,T]\times [0,L],\\
	\frac{1}{L}\int_0^L \phi^{T,m}(T, x)dx = m,\\
	\phi^{T,m}(0) \equiv B(m,T),
\end{cases}
\end{equation}
with periodic boundary conditions and where $B(m,T)$ is an unknown number depending only on $f$, $m$, and $T$.  It is simple to
check that the system \cref{e_inhomogeneous_logistic_tvp} is well
defined and is continuous in $m$ and $T$.  It is also easy to check
that $B(m,T)$ is increasing in $m$ and decreasing in $T$.

We now formulate our results precisely.
\begin{thm}\label{p_inhomogeneous}
Suppose that $u$ solves \cref{e_kpp} and that $f$ is an $L$-periodic function satisfying the conditions in \cref{e_f}.  Suppose $u_0$ satisfies the conditions in \cref{e_superexponential,e_left_limit,e_monotonic,e_regular}.  For any $r>0$,  we have that
\begin{equation}\label{e_mean_upper_bound}
	\limsup_{T\to\infty} \sup_{x \geq u_0^{-1}\left(B(m,T)\right) - rT} \fint_x^{x+L} u(T,y) dy \leq m,
\end{equation}
and
\begin{equation}\label{e_mean_lower_bound}
	\liminf_{T\to\infty} \inf_{x \leq u_0^{-1}\left(B(m,T)\right)+ rT} \fint_x^{x+L} u(T,y) dy \geq m,
\end{equation}
where $B(m,T)$ is defined in \eqref{e_inhomogeneous_logistic_tvp}.
\end{thm}
This result gives an analogous result to the \cref{p_homogeneous}.
Namely, that to find an interval $[nL, nL+L]$ such that $u(t,x)$ take
the value $m$ on average, one has to choose $n$ such that $u_0(x)$
roughly looks like $B(m,T)$ for $x \in [nL, nL + L]$.

In \cref{p_inhomogeneous}, the $B(m,T)$ term is a bit mysterious.  In
the homogeneous case, we characterize the location of the level sets in
terms of
the global-in-time solution to the logistic equation, an ordinary differential equation.  The following proposition does the same here, making rigorous the ODE-like behavior of $u$ by characterizing $B(m,T)$ in terms of solution to a global-in-time analogue to the logistic equation.
\begin{prop}\label{p_inhomogeneous_global}
Let the conditions of \cref{p_inhomogeneous} be satisfied.  Let $\psi_0$ and $f_0$ be the unique positive solutions to the eigenvalue problem
\[\begin{cases}
	(\psi_0)_{xx} + f_u(x,0) \psi_0 = f_0 \psi_0, ~~~x\in\T,\\
	\int \psi_0^2 dx = 1.
\end{cases}\]
Let $T_m$ be the unique time that $\fint \varphi(T_m) dx = m$.  Then
\[
	\lim_{T\to\infty} \frac{B(m,T) \int \psi_0(x)dx }{ \int \varphi(T_m - T,x) \psi_0(x) dx} = 1.
\]
\end{prop}

Again we may obtain an easy corollary regarding the movement of the average-level sets $\overline{E}_m(t)$.
\begin{cor}\label{p_inhomogeneous_level_sets}
Let the hypotheses of \cref{p_inhomogeneous_global} be satisfied.  For any $\epsilon>0$, and any $m\in (0,1)$ there is a time $T_0$, depending on all parameters, such that if $T\geq T_0$, then
\[
	\overline{E}_m(t)
		\subset \left[ u_0^{-1}\left( \varphi(T_{m+\epsilon} - T)\right), u_0^{-1}\left( \varphi(T_{m-\epsilon} - T)\right) \right].
\]
\end{cor}

As a part of the proof, we will show that there is a positive constant
$\alpha>0$ such that, for very large negative times, $\varphi(t,x)
\sim \alpha e^{f_0 t} \psi_0(x)$.  Hence \cref{p_inhomogeneous_global}
and \cref{p_inhomogeneous_level_sets} imply that there is a constant
$c_m$, depending only on $m$ and $f$, such that
\[
	\overline{E}_m(t) \sim u_0^{-1} \left( c_m e^{-f_0 t}\right)
\]
holds for large times $t$.  The constant $c_m$ is, in principle,
computable given a specific non-linearity $f$ by solving the global in
time periodic problem.

Finally, we prove a result which is analogous to the asymptotic
flatness result of \cite{HamelRoques}. This final result shows that
$u$ grows exactly as the global in time solution of the periodic
problem.
\begin{thm}\label{p_inhomogeneous_flatness}
Let the conditions of \cref{p_inhomogeneous} be satisfied.  
Let $\psi_0$ be as in \cref{p_inhomogeneous_global}.  Define
\begin{equation}\label{e_sn}
	S_{n} = \left( \int \psi_0(x)dx\right)^2 u_0(nL).
\end{equation}
Then
\[
	\lim_{T\to\infty} \max_{n \in \Z} \|u(T,\cdot) - \varphi(T^{S_n} + T, \cdot)\|_{L^\infty([nL, nL+L])} = 0.
\]
\end{thm}

\subsubsection*{Related work}
%

We will now briefly describe related work.  As we have mentioned above, 
the most closely related work is that of Hamel and Roques~\cite{HamelRoques}.  In this novel article, the authors present three main results regarding the solutions to the homogeneous Fisher-KPP equation (i.e. \eqref{e_kpp} with $f$ depending only on $u$).  First, they show that $\|u_x/u\|_{L^\infty(\R)} \to 0$ as $t$ tends to infinity.  Further, they show that if the initial data decays slower than $x^{-\alpha_f}$, for some $\alpha_f$ depending only on the smoothness of $f$, there are constants $C_m$ and $C$ such that, for $t$ large enough,
\[
	E_m(t) \subset\left[ u_0^{-1}(C_m \exp - f'(0)t),  u_0^{-1}(C m \exp - f'(0)t)\right]
.\]
We note that $\alpha_f = 2$ when $f(u) = u(1-u)$.  Lastly, for more general initial data, they obtain the exponential order.  Precisely, they show that, for any $\epsilon$,
\[
	E_m(t) \subset \left[u_0^{-1}\left(\exp - (f'(0) - \epsilon)t\right)
				, u_0^{-1}\left(\exp - (f'(0) + \epsilon)t\right)\right],
\]
holds for any $t$ large enough.

Hamel and Roques obtain their bounds on $E_m(t)$ by carefully building sub- and super-solutions using the function, $u_0(x) e^{ f'(0) t}$, and using the fact that we may bound $f$ below by a second order Taylor approximation at zero.  Our method involves quite different sub- and super-solutions, and the difference in precision between the constant in their bounds and our bounds, in \cref{e_kpp_speed}, can be attributed to the fact that our sub- and super-solutions feel the full affect of the non-linearity, $f$.

Beyond the work of Hamel and Roques, there has been recent interest in
super-linear in time propagation phenomena, or accelerating fronts.
In a series of articles~\cite{CabreRoquejoffre,CCR,CoulonRoquejoffre},
Cabr\'e, Coulon, and Roquejoffre investigated the phenomenon of
exponential spreading of level sets of solutions to the Fisher-KPP
with the fractional Laplacian.  In both cases, the mechanism is the same: when the data is spread out, $u(\cdot, x)$ evolves according to the logistic equation for every $x$.  In this model, the fat tails of the fractional Laplacian are responsible for spreading $u$ out in finite time.  We note that our techniques can be adapted
to yield analogous results when $u_0$ decays slower than the critical
cases investigated by Cabr\'e, Coulon, and Roquejoffre.  These accelerating fronts have also been
observed in biological models with a non-local dispersal operator with
fat tails~\cite{Garnier, KLV, MedlockKot} and in a kinetic
reaction-transport equation~\cite{BouinCalvezNadin}.  Finally, they
have been conjectured to occur in a reaction-diffusion-mutation model
with variable motility~\cite{BCMetal}.
%
%
%
%
%

\subsubsection*{Plan of the paper}

The results of \cref{p_homogeneous} are a special case of \cref{p_inhomogeneous}, and, hence, we only prove \cref{p_inhomogeneous}.  We do so in \cref{s_inhomogeneous} by creating sub- and super-solutions using the heat equation and \cref{e_inhomogeneous_logistic_tvp}, with carefully chosen initial data for both.  Then we prove \cref{p_inhomogeneous_global} in \cref{s_global_in_time} through an eigenfunction decomposition.  This characterizes the location of the level sets in terms of the global in time problem \cref{e_logistic_global}.  In \cref{s_convergence}, we prove \cref{p_inhomogeneous_flatness}, showing that at large times $u$ locally looks like the global in time solution to \cref{e_logistic_global}.  Finally, we show that \cref{e_logistic_global} admits a unique global in time solution in \cref{s_wellposed}.

\subsubsection*{Notation}
In order to make the discussion clearer, we adopt the following notation.  The constant, $C$, denotes an arbitrary constant independent of time which may change line-by-line.  In addition, we occasionally use $o(1)$ to mean a constant that tends to zero along with some quantity which we will specify.  We will make clear, in each usage, the dependencies of $C$ and $o(1)$.

\section{Spreading of the level-average sets}\label{s_inhomogeneous}

In this section, we prove \cref{p_inhomogeneous} by creating sub- and super-solutions by decomposing $u$ into a product of purely diffusive and purely reactive terms.  We utilize the Feynman-Kac formula along with our condition on the oscillations, \cref{e_regular}, in order to control the diffusive term.

In addition, we state a lemma which will allow us to translate the simple condition on the decay of $u_0$, \cref{e_regular}, to a more useful condition on the oscillations of $u_0$ very far to the right.
\begin{lem}\label{p_oscillations}
Suppose that $u_0\in C^1$ is strictly decreasing and satisfies
\[
	\lim_{x\to\infty} \frac{(u_0)_x(x)}{u_0(x)/\log(u_0(x))} = 0.
\]
Let $\lambda_t$ be some quantity for which there exist constants $c_1, c_2>0$ such that $\lambda_t \leq c_1 e^{-c_2 t}$.  Then, for any choice of $c_3$, we have that
\[
	\lim_{t\to\infty} \frac{u_0\left( u_0^{-1}(\lambda_t) \pm c_3 t\right)}{\lambda_t} = 1.
\]
The limit may be taken uniformly for any $c_1$ which is bounded above.
\end{lem}

\noindent We will use this with the choice $\lambda_t = B(m,t)$.  To this end, we need the following lemma that guarantees the exponential decay of $B(m,t)$ in time, which will be useful in its own right.

\begin{lem}\label{p_exponential}
Let $B(m,T)$ be defined as in \cref{p_inhomogeneous}.  Let $f_0$ be the largest eigenvalue of $\Delta + f_u(x,0)$.  Then $f_0>0$ and there exists positive constants $C$, independent of $m$, and $C_m$, depending on $m$, such that
\begin{equation}\label{e_exponential_1}
	\frac{m}{C} e^{-f_0 T} \leq B(m,T) \leq C_m e^{-f_0 T}.
\end{equation}
Moreover, there exists $\delta_0>0$ such that if $m \leq \delta_0$, then
\begin{equation}\label{e_exponential_2}
	B(m,T) \leq Cm e^{-f_0 T}.
\end{equation}
\end{lem}

\noindent We prove these lemmas in \cref{ss_lemmas}.

\subsection{An upper bound}\label{ss_upper_bound_inhomogeneous}

To begin, we define
\begin{equation}\label{e_xmt}
	x_m(T) = u_0^{-1}(B(m,T)),
\end{equation}
and we define a super-solution as $\overline{v} = \phi^{T,m} \overline{w}$ where $\phi^{T,m}$ satisfies \cref{e_inhomogeneous_logistic_tvp} and $\overline{w}$ satisfies
\[\begin{cases}
	\overline{w}_t = \overline{w}_{xx} + 2\left(\frac{\phi_x^{T,m}}{\phi^{T,m}}\right) \overline{w}_x,\\
	\overline{w}(0,x) = \max\left\{ \frac{u_0(x)}{B(m,T)}, 1\right\}.
\end{cases}\]
Indeed, we notice that $\overline{v}(0,x) \geq u_0(x)$ for every $x$ and that
\begin{equation}\label{e_supersolution}
\begin{split}
	\overline{v}_t - \overline{v}_{xx} - f(x,\overline{v})
		&= \phi^{T,m}_t \overline{w} + \phi^{T,m} \overline{w}_t - \phi^{T,m}_{xx}\overline{w} - 2 \phi_x^{T,m} \overline{w}_x - \phi^{T,m} \overline{w}_{xx} - f(x,\phi^{T,m}\overline{w})\\
		&= \left[\phi^{T,m}_{xx}  + f(x,\phi^{T,m})\right] \overline{w} + \phi^{T,m} \left[\overline{w}_{xx} + 2\frac{\phi_x^{T,m}}{\phi^{T,m}}\overline{w}_x\right] - \phi^{T,m}_{xx}\overline{w} - 2 \phi_x^{T,m} \overline{w}_x \\
			&~~~~ - \phi^{T,m} \overline{w}_{xx} - f(x,\phi^{T,m}\overline{w})\\
		&= \overline{w}\phi^{T,m}\left[\frac{f(x,\phi^{T,m})}{\phi^{T,m}} - \frac{f(x,\phi^{T,m}\overline{w})}{\overline{w} \phi^{T,m}}\right] \geq 0
.\end{split}
\end{equation}
The inequality in the last line follows from the fact that $\overline{w} \geq 1$, by the maximum principle, and from the fact that $f(x,s)/s$ is decreasing in $s$ by \cref{e_f}.

As a result of the above decomposition, we need only show that $\overline{w}$ does not deviate much from $1$ near $x_m(T)$.  To this end, we will use the Feynman-Kac representation of $\overline{w}$, see e.g. \cite{Bass}.  First, it follows from the standard parabolic regularity theory that
\begin{equation}\label{e_advection}
	\|\phi^{T,m}_x/\phi^{T,m}\|_\infty \leq C_0,
\end{equation}
where $C_0$ is a constant independent of $t$, $T$, and $m$, see e.g.~\cite{KrylovHolder,KrylovSobolev}.  We define a random process $X_t$ as the solution to
\begin{equation}\label{e_sde}
\begin{cases}
	dX_t = 2 \frac{\phi^{T,m}_x(T - t, X_t)}{\phi^{T,m}(T-t,X_t)}\ dt + \sqrt{2}\ dB_t,\\
	X_0 = x,
\end{cases}
\end{equation}
where $B_t$ is a standard Brownian motion.  We may then represent $w$ as
\[
	\overline{w}(t,x)
		= \E^x \left[\overline{w}(0,X_t)\right]
.\]
Fix $r_1>0$, to be determined later, and any $r_2$.  For $x \geq x_m(T) - r_2T$, we estimate this at time $T$ as
\begin{equation}\label{e_decomposition_upper}
	\overline{w}(T,x)
		= \E^x \left[\overline{w}(0,X_T)\1_{\{X_T \leq x - r_1T\}}\right] + \E^x \left[(\overline{w}(0,X_T) - 1)\1_{\{X_t \geq x - r_1T\}}\right] + \E^x \left[\1_{\{X_t \geq x - r_1t\}}\right].
\end{equation}
First, we may use \cref{p_exponential} to find a constant $C$ with which we may bound $\|\overline{w}(0)\|_\infty \leq C e^{ f_0 T}$.  Then, if $r > 2C_0 + 2\sqrt{f_0}$, with $C_0$ as in \cref{e_advection}, the first term in \cref{e_decomposition_upper} may be bounded as
\[\begin{split}
	\E^x \left[\overline{w}(0,X_T)\1_{\{X_T \leq x - rT\}}\right]
		&\leq C e^{f_0T} \PP\left\{ X_T \leq x - rT\right\}
		\leq C e^{f_0T} \PP\left\{ \sqrt{2}B_T - 2C_0 T \leq - rT\right\}\\
		&\leq C e^{f_0T} \PP\left\{ \sqrt{2} B_T \leq - 2\sqrt{f_0} T \right\}
		\leq \frac{C}{\sqrt{T}}
.\end{split}\]
The last inequality follows from a standard estimate on the tail of a Gaussian.  Thus, the first term tends to zero as $T$ tends to infinity.  The third term in \cref{e_decomposition_upper} can be shown to converge to $1$ in the exact same manner as above.  Notice that these limits did not depend on $m$.

Hence, we need only show that the second term converges to zero.  To this end we bound the second term as
\[\begin{split}
	\E^x \left[(\overline{w}(0,X_T) - 1)\1_{\{X_t \geq x - r_1T\}}\right]
		&\leq \E^x \left[\frac{u_0(x - r_1T) - u_0(x_m(T))}{B(m,T)} \right]\\
		&\leq \E^x \left[\frac{u_0(x_m(T) - (r_1 + r_2)T) - u_0(x_m(T))}{B(m,T)} \right].
\end{split}\]
Using the definition of $x_m(T)$, \cref{e_xmt}, this tends to zero as a consequence of \cref{p_oscillations} and \cref{p_exponential}.  We notice that this limit can be taken independent of $m$ if $m$ is bounded away from $1$ as a consequence of the tracking the dependence on $m$ through \cref{p_oscillations} and \cref{p_exponential}.

Hence we obtain that if $x\geq x_m(T) - r_2T$, then $\overline{w}(T,x) \leq 1 + o(1)$, where $o(1)$ tends to zero as $T$ tends to infinity.  This yields
\begin{equation}\label{e_upper_bound}
	u(T,y)
		\leq \phi^{T,m}(T,y) \overline{w}(T,y)
		\leq \phi^{T,m}(T,y) [1 + o(1)]
.\end{equation}
By construction, $\fint \phi^{T,m}(T,y)dy = m$.  Hence, averaging over $[x,x+L]$ gives us the upper bound \cref{e_mean_upper_bound}.

We note that, to take the limits above, we used only estimates on the Gaussian and \cref{p_oscillations} and \cref{p_exponential}.  Hence, these limits can be taken independent of $m$ as long as $m$ is uniformly bounded away from $0$ and $1$.

\subsection{A lower bound}\label{ss_lower_bound_inhomogeneous}

The argument here is similar to the argument in \cref{ss_upper_bound_inhomogeneous}, above.  As before we fix $T>0$ and $m\in(0,1)$, define $x_m(T)$ as in \cref{e_xmt}, and we $\underline{v}(t,x) = \phi^{T,m}(t,x)\underline{w}(t,x)$ where $\phi^{T,m}(t,x)$ is defined in \cref{e_inhomogeneous_logistic_tvp}.  Here, we instead define $\underline{w}$ as the solution to
\[\begin{cases}
	\underline{w}_t = \underline{w}_{xx} + 2\frac{\phi^{T,m}_x}{\phi^{T,m}}\underline{w}_x,\\
	\underline{w}(0,x) = \min\left\{ \frac{u_0(x)}{B(m,T)}, 1\right\}.
\end{cases}\]
Notice that $\underline{v}(0,x) \leq u_0(x)$ for all $x$.  The computation above, \cref{e_supersolution}, shows that
\[
	\underline{v}_t - \underline{v}_{xx} - f(x,\underline{v})
		= \underline{w}\phi^{T,m}\left[\frac{f(x,\phi^{T,m})}{\phi^{T,m}} - \frac{f(x,\phi^{T,m}\underline{w})}{\underline{w} \phi^{T,m}}\right] \leq 0,
\]
where the last inequality follows since $\underline{w}\leq 1$, as a consequence of the maximum principle, and since $f(x,s)/s$ is decreasing for $s\in [0,1]$.

Again, defining $X_t$ by \cref{e_sde}, we obtain that
\[
	\underline{w}(t,x) = \E^x\left[ \underline{w}(0, X_t) \right].
\]
Fix $r_1>0$, to be determined below, and any $r_2$.  For any $x \leq x_m(T) + r_2 T$, we decompose $w(T,x)$ as
\[
	\underline{w}(T,x)
		= \E^x \left[\underline{w}(0,X_T)\1_{\{X_T \geq x + r_1T\}}\right] + \E^x \left[(\underline{w}(0,X_T) - 1)\1_{\{X_t \leq x + r_1T\}}\right] + \E^x \left[\1_{\{X_t \leq x + r_1T\}}\right].
\]
For $x\leq x_m(T) + r_2T$, we may again choose $r_1$ large enough that the first term on the right hand side tends to zero as $T$ tends to infinity and that the last term on the right tends to one as $T$ tends to infinity.  In addition, \cref{p_oscillations} and \cref{p_exponential} imply that the second term on the right tends to zero as $T$ tends to infinity.  Hence, we obtain, for any $x \leq x_m(T) + r_2T$,
\begin{equation}\label{e_lower_bound}
	u(T,x)
		\geq \phi^{T,m}(T,x) \underline{w}(T,x)
		\geq \phi^{T,m}(T,x) (1 - o(1))
.\end{equation}
Using that, by construction, $\fint \phi^{T,m}(T,y)dy = m$, integrating over $[x-L, x]$ yields the lower bound in \cref{e_mean_lower_bound} for any $x \leq x_m(T) + r_2 T$.  Again, we note that the limit in \cref{e_lower_bound} may be taken uniformly for $m$ bounded away from $1$.

\subsection{Controlling the oscillations}\label{ss_lemmas}

In this section we prove \cref{p_oscillations} and \cref{p_exponential}.  We first prove \cref{p_oscillations}, by using an ODE argument to compare $u_0$ to a sequence of equations which decay at the critical order but with increasingly slower rates.

\begin{proof}[Proof of \cref{p_oscillations}]
We will prove this lemma for $u_0^{-1}(\lambda_t) - c_3 t$, though the proof is exactly analogous in the case of $u_0^{-1}(\lambda_t) + c_3 t$.  Fix $\epsilon > 0$, and define $h_\epsilon(x) = e^{-\sqrt{\epsilon x}}$.  Notice that
\begin{equation}\label{e_h_ode}
	h' = -\frac{\epsilon}{2} \frac{h}{\log(h)}
.\end{equation}
Define 
\[
	x_0 = u_0^{-1}(\lambda_t)
	~~~\text{ and } ~~~
	y_0 = h^{-1}(\lambda_t)
.\]
We may then shift $h$ to obtain $g(x) = h(x - x_0 + y_0)$.  Notice that $g(x_0) = u_0(x_0)$.  The condition on the decay of $\lambda_t$ and the fact that $u_0$ decays slower than any exponential implies that $u_0^{-1}(\lambda_t) - c_3 t$ tends to infinity.  Hence we may choose $t$ large enough, uniformly in $c_1$ if $c_1$ is bounded above, such that if $x\geq x_0 - c_3 t$, then
\begin{equation}\label{e_u_0_ode}
	0 \geq u_0'(x) \geq -\frac{\epsilon}{2} \left|\frac{u_0(x)}{\log(u_0(x))}\right|.
\end{equation}
We note that $u_0$ is a supersolution to \cref{e_h_ode} as a consequence of \cref{e_u_0_ode}.  Hence $u_0(x) \leq g(x)$ for all $x_0 - c_3 t \leq x \leq x_0$.  Using this, along with the fact that $u_0(x_0) = g(x_0)$, we obtain
\begin{equation}\label{e_ordering}
	0
		\leq u_0(x_0 - c_3 t) - u_0(x_0)
		\leq u_0(x_0 - c_3 t) - g(x_0)
		\leq g(x_0 - c_3 t) - g(x_0).
\end{equation}
To conclude, we estimate the last term, $g(x_0 - c_3 t) - g(x_0)$.  Returning to our function $h$ and manipulating the equation, we obtain
\[\begin{split}
	g(x_0 - c_3 t) - g(x_0)
		&= h(y_0 - c_3 t) - h(y_0)\\
		&= \exp\{ -\sqrt{\epsilon(y_0 - c_3 t)}\} - \exp\{-\sqrt{\epsilon y_0}\}\\
		&= \exp\{ - \sqrt{\epsilon y_0}\} \left[\exp\left\{ \sqrt{\epsilon y_0} \left(1 - \sqrt{1 - (c_3t/y_0)}\right)\right\} - 1\right]
.\end{split}\]
It is a simple consequence of Taylor's theorem that there is a constant $C$ such that
\[
	1 - \sqrt{1 - (c_3 t/y_0)}\leq C( c_3 t / y_0)
.\]
Applying this in the equation above, and using that, by construction, $e^{\sqrt{\epsilon y_0}} = \lambda_t$, we obtain
\[
	g(x_0 - c_1 t) - g(x_0) \leq \lambda_t \left[\exp\left\{ \frac{ C\sqrt{\epsilon} c_3 t}{\sqrt{y_0}}\right\} - 1 \right].
\]
Then using that $\lambda_t \leq c_1 e^{- c_2 t}$, we may solve for $y_0$ to obtain
\[
	y_0 \geq \frac{1}{\epsilon^2} \left(c_2 t - \log(c_1)\right)^2.
\]
Returning to \cref{e_ordering}, and using that $u_0(x_0) = \lambda_t$, by construction, we obtain
\[
	0 \leq u_0(x_0 - c_3 t) - \lambda_t \leq g(x_0 - c_3 t) - g(x_0) \leq \lambda_t \left[\exp\left\{\epsilon  \frac{ C c_3 t}{c_2 t - \log(c_1)}\right\} - 1 \right].
\]
Adding $\lambda_t$ to to all sides of each inequality and then dividing by $\lambda_t$ yields
\[
	1 \leq \frac{u_0(x_0 - c_3 t)}{\lambda_t} \leq \exp\left\{ C' \epsilon\right\},
\]
where $C'$ is a universal constant which bounds $C c_3 t / (c_2 t - \log(c_1))$.  Again, $C'$ can be chosen uniformly in $c_1$ as long as $c_1$ is bounded above.  The result then follows by using the fact that we can choose $\epsilon$ as small as we would like by taking $t$ larger.
\end{proof}

We will now prove \cref{p_exponential} in order to make \cref{p_oscillations} applicable.  The idea of the proof is to use the linearized equation to control the growth of $B(m,T)$.

\begin{proof}[Proof of \cref{p_exponential}]
First, we verify that $f_0$ is positive.  To see this, we simply use the Rayleigh quotient characterization of $f_0$ as
\[
	f_0 = \max_{\psi\in H^1(\mathbb{T})} \frac{\int \left(f_u(x,0) \psi(x)^2 - |\nabla \psi(x)|^2\right)dx}{\int_{\mathbb{T}} \psi(x)^2 dx},
\]
where $\mathbb{T} = [0,L]$ with the ends identified.  One can verify the positivity of $f_0$ by testing the equation with the function $\psi \equiv 1$ and using the positivity of $f_u(x,0)$.

Fix $T>0$, and let $\psi_0$ be the unique eigenfunction associated with $f_0$ with $\|\psi_0\|_2 = 1$.  It follows from the Krein-Rutman theorem that this function exists and is positive~\cite{DAUTRAY LIONS}.  Notice that $\overline{\phi}(t,x) = \alpha \psi_0(x) e^{f_0 t}$ satisfies
\[
	\overline{\phi}_t
		= \Delta \overline{\phi} + f_u(x,0)\overline{\phi}
		\geq \Delta \overline{\phi} + f(x,\overline{\phi}).
\]
Hence, if $\alpha \psi_0(x) \geq B(m,T)$, we have that $\overline{\phi}(t,x) \geq \phi^{T,m}(t,x)$ for all $x$ and all $t$, where $\phi^{T,m}$ is defined by \cref{e_inhomogeneous_logistic_tvp}.  Choosing
\[
	\alpha = \frac{B(m,T)}{\min_x \psi_0(x)},
\]
this implies that
\[
	B(m,T) \frac{\|\psi_0\|_\infty}{\min_x \psi_0(x)} e^{f_0 T}
		= \alpha \|\psi_0\|_\infty e^{f_0 T}
		\geq \fint \overline{\phi}(T,x)dx
		\geq \fint \phi^{T,m}(T,x)dx
		= m
.\]
Rearranging this inequality yields the lower bound on $B(m,T)$.

The upper bound is similar, but requires slightly more work.  We first obtain a bound for small $m$ by using a modulation technique found in \cite{HNRR12,HNRR13}.  Notice that our conditions on $f$ guarantee the existence of $M$ such that
\begin{equation}\label{e_concave}
	f(x,s) \geq f_u(x,0)s - M s^{1+\delta},
\end{equation}
holds for all $s \in [0,1]$.  Define $\underline{\phi}(t,x) = \alpha(t) \psi_0(x) e^{f_0 t}$.  We find conditions on $\alpha$ to make $\underline{\phi}$ a sub-solution of $\phi$.  To this end, we compute
\[\begin{split}
	\underline{\phi}_t - \underline{\phi}_{xx} - f(x,\underline{\phi})
		&= \frac{\alpha'}{\alpha} \underline{\phi} + f_0 \underline{\phi} - \alpha (\psi_0)_{xx} e^{f_0 t} - f(x,\underline{\phi})\\
		&=\frac{\alpha'}{\alpha} \underline{\phi} + f_0 \underline{\phi} - \alpha [f_0 \psi_0 - f_u(x,0)\psi_0] e^{f_0 t} - f(x,\underline{\phi})\\
		&= \frac{\alpha'}{\alpha}\underline\phi + f_0 \underline\phi - f_0 \underline\phi + f_u(x,0) \underline\phi - f(x,\underline\phi)\\
		&\leq \frac{\alpha'}{\alpha}\underline\phi + M \underline\phi^{1+\delta}
\end{split}\]
Hence $\underline\phi$ is a sub-solution if we choose
\[
	\frac{\alpha'}{\alpha^{1+\delta}} = - M e^{f_0 \delta t} \|\psi_0\|_\infty^\delta ~~~ \text{ and } ~~~ \alpha(0) = \frac{B(m,T)}{\|\psi_0\|_\infty}.
\]
In other words, we may let
\[
	\alpha(t) = \frac{B(m,T)}{\|\psi_0\|_\infty} \left[\frac{f_0 \delta}{ B(m,T)^\delta M [ e^{f_0 \delta t} - 1] + f_0 \delta} \right]^{1/\delta}
\]
Hence at time $t = T$, we may average $\underline\phi$ and use that $\underline\phi$ sits below $\phi^{T,m}$ to obtain
\[
	\frac{B(m,T)}{\|\psi_0\|_\infty} \left[ \frac{f_0 \delta}{B(m,T)^\delta M[e^{f_0 \delta T} - 1] + f_0 \delta} \right]^{1/\delta} e^{f_0 T} = \fint \underline\phi(T,x) dx \leq \fint \phi^{T,m}(T,x)dx = m.
\]
Re-arranging this we see that
\[
	f_0 \delta B(m,T)^{\delta} e^{f_0 \delta T} \leq m^\delta \|\psi_0\|_\infty^\delta \left[ B(m,T)^\delta e^{f_0 \delta T} M + f_0 \delta\right].
\]
If $m^\delta M \|\psi_0\|_\infty^\delta  \leq f_0 \delta/2$, then we may rearrange the above to obtain
\begin{equation}\label{e_bmt_upper}
	B(m,T)^{\delta} \leq 2m^\delta \|\psi_0\|_\infty^\delta  e^{-f_0 \delta T},
\end{equation}
finishing the claim for small $m$.

The argument is somewhat more complicated for general $m$.  We prove it in two steps.  First, we show that there is $R_1$, which is uniformly bounded if $m$ is uniformly bounded away from $0$ and $1$, such that $B(m/2, T - R_1) \geq B(m, T)$.  To see this, we take $C_{f,m}$ to be a positive constant, arising from parabolic regularity and the smoothness and positivity of $f$, such that $f(x, \phi^{T,m}(t,x)) \geq C_{f,m}$ if the average of $\phi$ is between $m/2$ and $m$.  Then we choose $R_1 = m / (2 C_{f,m})$ and we will show by contradiction that this is the correct choice of $R_1$.  If not, and instead $B(m/2, T-R_1) < B(m,T)$ holds, then the maximum principle gives that
\begin{equation}\label{e_mean_bound}
	\frac{m}{2}
		= \fint \phi^{T-R_1, m/2}(T-R_1,x)dx
		< \fint \phi^{T,m}(T-R_1, x) dx.
\end{equation}
Using our choice of $C_{f,m}$ above, we may integrate the equation for $\phi^{T,m}$ in space for any $t \in [T-R_1, T]$ to obtain
\[
	\partial_t \fint \phi^{T,m}(t,x) dx = \fint f(x,\phi^{T,m}) dx \geq C_{f,m}.
\]
Integrating this in time from $T-R_1$ to $T$, using the bound in \cref{e_mean_bound}, and using our choice of $R_1$ yields
\[
	m = \fint \phi^{T,m}(T,x) dx
		\geq C_{f,m} R_1 + \fint \phi^{T,m}(T-R_1,x) dx
		> \frac{m}{2} + \frac{m}{2},
\]
which is clearly a contradiction.

Having proven this fact, we now bootstrap the result for small $m$ to obtain the statement for general $m$.  For any $m$, we may find $N$ such that $m/ 2^N$ is small enough to apply the work above and obtain \cref{e_bmt_upper}.  Indeed, we choose $N$ large enough that
\[
	\left(\frac{m}{2^N}\right)^\delta M \|\psi_0\|_\infty^\delta  < f_0 \delta/2,
\]
and as a result obtain that, for any $\tilde T> 0$,
\begin{equation}\label{e_bmt_upper2}
	B\left(\frac{m}{2^N}, \tilde T\right) \leq C \left( \frac{m}{2^N}\right) e^{-f_0 \tilde T}.
\end{equation}
In addition to this, we may iterate the procedure above to find $R_1, R_2, \dots, R_N$ such that
\[
	B\left( \frac{m}{2^i}, T - (R_1 + \cdots + R_i)\right)
		\leq B\left( \frac{m}{2^{i+1}}, T - (R_1 + \cdots + R_{i+1})\right)
\]
holds for any $i$, and hence,
\begin{equation}\label{e_iterate}
	B(m, T)
		\leq B\left(\frac{m}{2^N}, T - (R_1 + \cdots R_N)\right)
.\end{equation}
Combining \cref{e_bmt_upper2,e_iterate} and defining $C_m = C m 2^{-N} e^{-f_0(R_1 + \cdots R_N)}$, we obtain
\[
	B(m,T)
		\leq C \frac{m}{2^N} e^{-f_0 (T - R_1 - R_2 - \cdots - R_N)}
		= C_m e^{-f_0 T},
\]
where we have defined $C_m = C m 2^{-N} e^{f_0(R_1 + \cdots+R_N)}$.  This concludes the proof.

\end{proof}

\section{Characterization of the speed}\label{s_global_in_time}

In this section, we will show how to characterize $B(m,T)$ in terms of the the global-in-time solution $\varphi$ of \cref{e_logistic_global}.  In the homogeneous case, there is nothing to prove since $\varphi$ and $\varphi^{T,m}$ are equal up to a translation in time, by the uniqueness of solutions of ordinary differential equations.  In the inhomogeneous setting, we need to deal with the fact that $\varphi$ is not necessarily flat, while $\phi^{T,m}(0,x) \equiv B(m,T)$.  The idea of the proof is to run the system for time $T/2$ and use spectral estimates to show that the $\varphi(T_m - T/2 + t)$ and $\phi^{T,m}(T/2 + t)$ will be close.  From then on, they must remain close.

Before we begin, we will need one fact about the global in time solution.  We delay the proof of this lemma until \cref{s_wellposed} as it will be crucial in the proof of the well-posedness of \cref{e_logistic_global}.
\begin{lem}\label{p_exponential_global}
Let $\varphi$ be any solution of \cref{e_logistic_global}.  Then there exist positive constants $\alpha$ and $\omega$ such that
\[
	\lim_{t\to-\infty} \frac{\varphi(t,x)}{\psi_0(x) e^{f_0 t}} = \alpha,
		~~\text{ and }~~
		\lim_{t\to\infty} \frac{1 - \varphi(t,x)}{\psi_1(x) e^{-f_1 t}} = \omega
\]
where $\psi_0$ and $f_0$, defined in \cref{p_inhomogeneous_global}, are the normalized eigenfunction and principle eigenvalue of $\Delta + f_u(x,0)$, and where $\psi_1$ and $f_1$ are the normalized eigenfunction and principle eigenvalue of $\Delta - f_u(x,1)$.
\end{lem}

Now, our first step is to prove the following lemma.
\begin{lem}\label{p_linearized_error}
Suppose that $\gamma_0$ is a $C^\infty(\mathbb{T})$ function such that $\|\gamma_0\|_\infty \leq C e^{- f_0T}$ for some $C>0$.
\[\begin{cases}
	\gamma_t = \gamma_{xx} + f(x, \gamma),\\
	\gamma(0,x) = \gamma_0(x).
\end{cases}\]
Fix $\epsilon>0$.  Then there is $T_0 >0$, which can be chosen uniformly if $C$ is bounded above, such that if $T \geq T_0$, we have that
\[
	\left\|\gamma(T/2) - \psi_0 \left(\int \gamma_0 \psi_0\right) e^{f_0 T/2}\right\|_2
		\leq \epsilon e^{-f_0 T/2}.
\]
\end{lem}
\begin{proof}
Define
\[
	\Gamma(t,x) = \gamma(t,x) e^{f_0 (T-t)},
\]
and we can easily see that $\Gamma$ satisfies
\[
	\Gamma_t - \Gamma_{xx} + f_0\Gamma = f(x,\gamma) e^{f_0(T-t)} = \left(\frac{f(x,\gamma)}{\gamma} \right) \Gamma.
\]
Letting $\Gamma_1(t) = \int \Gamma(t,x) \psi_0(x) dx$, we notice that
\[
	\dot{\Gamma}_1
		= e^{f_0(T-t)} \int \left( f(x,\gamma) - f_u(x,0)\gamma\right)\psi_0(x) dx.
\]
First, we notice that $\dot{\Gamma}_1 \leq 0$.  In addition, since $\psi_0$ multiplied by a large constant is a super-solution of $\Gamma$, it follows that $\Gamma$ is uniformly bounded.  We utilize this, along with \cref{e_concave}, to obtain
\[\begin{split}
	0 \geq \dot{\Gamma}_1
		&\geq - M e^{f_0(T-t)} \int \gamma(x)^{1+\delta} \psi_0(x) dx
		\geq - M C'e^{f_0(T-t)} e^{\delta f_0 (t-T)} \int \gamma(x) \psi_0(x)dx\\
		&\geq - M C' e^{\delta f_0 (t - T)} \Gamma_1.
\end{split}\]
Dividing by $\Gamma_1$, integrating in time, and exponentiating yields
\[
\begin{split}
	1 \geq \frac{\Gamma_1(T/2)}{\Gamma_1(0)} \geq \exp\left\{ - e^{ - \frac{\delta f_0}{2} T} C''\left[ 1 - e^{-\frac{\delta f_0}{2}T} \right]\right\}.
\end{split}
\]
Defining $\gamma_1(t) = \int \gamma(t,x) \psi_0(x)dx$, we may rewrite this inequality as
\begin{equation}\label{e_gamma_1_bound}
	| \gamma_1(T/2) - e^{f_0 T/2} \gamma_1(0)| \leq o(1) e^{-f_0T/2}.
\end{equation}

Now we let $\Gamma_2(t,x) = \Gamma(t,x) - \Gamma_1(t) \psi_0(x)$, and we bound $\Gamma_2$.  Notice that $\Gamma_2$ satisfies the equation
\begin{equation}\label{e_gamma_2}
\begin{split}
	(\Gamma_2)_t - &(\Gamma_2)_{xx} - (f_u(x,0) - f_0) \Gamma_2\\
		&= \left[ f(x,\gamma) e^{f_0(T-t)} - f_u(x,0)\Gamma\right] - \psi_0(x) e^{f_0(T-t)} \int (f(y,\gamma) - f_u(y,0)\gamma(y)) \psi_0(y) dy 
\end{split}
\end{equation}
Because $\Gamma_2$ is orthogonal to the the eigenspace for eigenvalue $f_0$ of $\Delta + f_u(x,0)$, then there is $\alpha$ such that
\[
	\int \left[ (\Gamma_2)_{xx} + (f_u(x,0) - f_0) \Gamma_2\right]\Gamma_2 dx
		\leq - \alpha \|\Gamma_2\|_2^2.
\]
Hence, multiplying \cref{e_gamma_2} by $\Gamma_2$ and, again, using that $\Gamma_2$ and $\psi_0$ are perpendicular yields
\[\begin{split}
	\frac{1}{2} \frac{d}{dt} \|\Gamma_2\|_2^2 + \alpha \|\Gamma_2\|_2^2
		&\leq e^{f_0(T-t)} \int \left[ f(x,\gamma) - f_u(x,0)\gamma \right] \Gamma_2 dx\\
		&\leq M e^{f_0(T-t)} \int \gamma^{1+\delta} |\Gamma_2| dx = M e^{\delta f_0(t-T)} \int \Gamma^{1+\delta} \Gamma_2 dx
\end{split}\]
As we noted above, $\|\Gamma\|_\infty$ is uniformly bounded in time.  Hence, we may write this as
\[
	\frac{d}{dt}\|\Gamma_2\|_2^2 + 2\alpha \|\Gamma_2\|_2^2
		\leq C e^{\delta f_0(t-T)} \|\Gamma_2\|_2,
\]
where $C$ is a new constant independent of $t$ and $T$.  Solving this differential inequality and using that $\|\Gamma_2(0)\|\leq C$ for some constant $C$, yields constants $C>0$ and $\theta>0$ such that $\|\Gamma_2(T/2)\|_2 \leq Ce^{-\theta T}$.

The combination of this inequality with \cref{e_gamma_1_bound} yields that
\[
	\|\gamma(T/2) - e^{f_0 T/2} \gamma_1(0) \psi_0\|_2
		\leq C|\gamma_1(T/2) - e^{f_0 T/2} \gamma_1(0)| + e^{-f_0 T/2} \|\Gamma_2\|_2
		\leq o(1) e^{-f_0 T/2},
\]
finishing the proof.
\end{proof}

We now show how to use this lemma to conclude \cref{p_inhomogeneous_global}.  Fix $m\in (0,1)$ and define the starting time $S(m,T)$ as the time such that
\begin{equation}\label{e_smt}
	\int \varphi(S(m,T),x) \psi_0(x) dx =  B(m,T) \int \psi_0.
\end{equation}
We first show that $\fint\varphi(S(m,T)+T,x)dx \to m$ as $T$ tends to infinity.  Then, defining $T_m$ as the unique time such that
\[
	\fint \varphi(T_m, x) dx = m,
\]
 we will leverage this to show that
 \begin{equation}\label{e_bmt_limit}
 	\lim_{T\to\infty} \frac{\int \varphi(T_m - T, x) \psi_0(x) dx}{B(m,T) \int \psi_0} = 1,
 \end{equation}
finishing the proof.

By parabolic regularity, \cref{p_linearized_error}, and \cref{p_exponential}, we have that
\[
	\|\varphi(S(m,T) + T/2) - \phi^{T,m}(T/2)\|_2 \leq o(1) e^{-f_0 T/2},
\]
where $o(1)$ tends to zero as $T$ tends to infinity (uniformly for $m$ bounded away from $1$).  By parabolic regularity this may be strengthened to
\[
	\|\varphi(S(m,T) + T/2) - \phi^{T,m}(T/2)\|_\infty \leq o(1) e^{-f_0T/2}.
\]
Define $\eta(t,x) = \varphi(S(m,T) + T/2 + t,x) - \phi^{T,m}(T/2 + t,x)$ and notice that $\eta$ solves
\[
	\eta_t - \eta_{xx} = c(t,x) \eta,
\]
where $c(t,x) = (f(x,\varphi^1) - f(x,\varphi^2))/(\varphi^1 - \varphi^2) \leq f_u(x,0)$.  Hence, $o(1) \psi_0 e^{f_0(t - T/2)}$ is a super-solution to the equation above with initial data larger than $\eta(0,x)$.  This implies that
\[
	|\eta(T/2,x)| \leq o(1) \psi(x) \leq o(1),
\]
holds for every $x$.  Using the definition of $\eta$, this implies that
\[
	\|\varphi(S(m,T)+T,\cdot) - \varphi^{T,m}(T,\cdot)\|_\infty \leq o(1).
\]
By our definition of $\varphi^{T,m}$, this implies that
\[
	\lim_{T\to\infty} \fint \varphi(S(m,T)+T,x) dx = m,
\]
holds, where the limit may be taken uniformly in $m$ if $m$ is bounded away from $1$.

Now we will show \cref{e_bmt_limit}.  Suppose this does not hold.  Let $\theta > 1$, without loss of generality, be such that there is a sequence $T_1, T_2, T_3, \dots$ such that
\[
	\lim_{n\to\infty} \frac{\int \varphi(T_m - T_n, x) \psi_0(x) dx}{B(m,T_n) \int \psi_0}
		> \theta.
\]
We note that the proof is similar for $\theta < 1$ and the limit above is less than $\theta$.  Using \cref{p_exponential_global} and the definition of $S(m,T)$, we can choose $N_0$ be such that if $n \geq N_0$ then
\[
	\frac{e^{T_m - T_n}}{e^{S(m,T_n)}}
		> \theta,
\]
and hence,
\[
		T_m - T_n - S(m,T_n) > \log(\theta) > 0.
\]
By parabolic regularity, along with the smoothness and positivity of $f$, we may find $\epsilon$ depending only on $m$ and $\theta$ such that $f(x,\varphi(t,x)) > \epsilon$ for all $t \in [T_n + S(m,T_n), T_m]$.  This holds because $\varphi$ is bounded away from $0$ and $1$ on this interval.  Integrating the equation for $\varphi$ on this time interval yields
\[
	\varphi(T_m) - \varphi(T_n + S(m,T_n)) = \int_{T_n + S(m,T_n)}^{T_m} f(x,\varphi) dx \geq \epsilon \log(\theta).
\]
The average of the first term is $m$, and the average of the second term converges to $m$ for large $n$.  This implies that the left hand side tends to zero, which is a contradiction.  This finishes the proof.

\section{Convergence to the global in time problem}\label{s_convergence}

In this section, we prove \cref{p_inhomogeneous_flatness}.  We will do this in two steps.  First we will prove the closeness of $u$ and $\phi^{T,m}$.  Then we will show that $\phi^{T,m}$ and $\varphi$ are close when suitably translated in time.  To this end, we state two lemmas.

\begin{lem}\label{p_convergence_phit}
Define $\phi^{T,m}$ as in \cref{e_inhomogeneous_logistic_tvp}, and define $I_{m,T}$ to be the interval $[nL, nL+L)$ containing $u_0^{-1}(B(m,T))$, we have that
\begin{equation}\label{e_uniform_convergence}
	\lim_{T\to\infty} \max_{m\in(0,1)} \left\| u(T,\cdot) - \phi^{T,m}(T,\cdot)\right\|_{L^\infty(I_{m,T})} = 0.
\end{equation}
\end{lem}

\begin{lem}\label{p_convergence_smt}
Define $S(m,T)$ as in \cref{e_smt}.  Then,
\[
	\lim_{T\to\infty} \max_{t\in [T/2,\infty]} \|\varphi(S(m,T)+t,\cdot) - \phi^{T,m}(t, \cdot)\|_\infty = 0.
\]
This limit can be taken uniformly for $m$ which is bounded away from $1$.
\end{lem}
We will prove \cref{e_uniform_convergence} at the end of this section.  The proof of \cref{p_convergence_smt} follows from our choice of $S(m,T)$ and from work in \cref{p_linearized_error}.  Hence, we omit it.  We now show how to conclude \cref{p_inhomogeneous_flatness} using these two lemmas.

\begin{proof}[Proof of \cref{p_inhomogeneous_flatness}]
First, denote $I_{m,T}$ as $[nL,nL+L]$.  Define $T^{S_n}$ as the unique time such that
\[
	\fint \varphi(T^{S_n},x)dx = S_n
,\]
where $S_n$ is as in \cref{e_sn}.  We show that if $T$ (or equivalently $n$) is large enough, then $T^{S_n}$ and $S(m,T)$ are close, and that we can obtain this in a uniform way for $m$ bounded away from 1.  To this end, notice that, by the definition of $I_{m,T}$,
\[
	u_0(nL) \geq B(m,T) \geq u_0(nL+L)
.\]
In addition, by \cref{p_oscillations} and \cref{p_exponential}, we have that
\begin{equation}\label{e_bmt_n}
	\frac{u_0(nL)}{B(m,T)} = 1 + o(1),
\end{equation}
where $o(1)$ tends to zero as $T$ tends to infinity in a uniform way if $m$ is bounded away from $1$.  Using \cref{p_exponential_global}, we then compute that
\[\begin{split}
	1 + o(1)
		&= \frac{u_0(nL)}{B(m,T)}
		= \left(\frac{\fint \varphi(T^{S_n},x)dx}{\left(\int \psi_0(x)dx\right)^2}\right)\left(\frac{\int \psi_0 dx}{ \int \varphi(S(m,T),x) \psi_0(x)dx}\right)\\
		&= \left(\frac{\fint \alpha (\psi_0(x) + o(1)) e^{f_0 T^{S_n}}dx}{\left(\int \psi_0(x)dx\right)^2}\right)\left(\frac{\int \psi_0 dx}{ \int  \alpha (\psi_0(x) + o(1)) e^{f_0 S(m,T)}\psi_0(x)dx}\right)
		= \frac{e^{f_0 T^{S_n}}}{e^{f_0 S(m,T)}} (1+o(1)).
\end{split}\]
Again, the constant $o(1)$ tends to zero uniformly as long as $m$ is bounded away from $1$.  This yields the claim that $T^{S_n} - S(m,T)$ tends to zero as $T$ tends to infinity.

Fix $\epsilon>0$.  Using \cref{p_convergence_phit}, choose $T_1$ large enough that if $T \geq T_1$, then
\begin{equation}\label{e_u_phi_close}
	\max_{m\in(0,1)} \|u(T,\cdot) - \phi^{T,m}(T,\cdot)\|_{L^\infty(I_{m,T})} \leq \epsilon / 2.
\end{equation}
In addition, choose $m_1< 1$, depending only on $f$ and $\epsilon$, such that if $m \geq m_1$, then
\[
	\phi^{T,m}(T, x) \geq 1 - \epsilon/2,
\]
for all $x \in \T$ and any $T$.

We will consider three cases.  First, if $n$ is small enough that $\varphi(T^{S_n}, x) \geq B(m_1,T)$, the maximum principle implies that
\[
	\varphi(T^{S_n} +T, x) \geq \phi^{T,m_1}(T,x) \geq 1- \epsilon/2.
\]
This, combined with \cref{e_u_phi_close}, implies that 
\[
	\|\varphi(T^{S_n}+T,x) - u(T,\cdot)\|_{L^\infty([nL, nL+L])} \leq \epsilon,
\]
for $T \geq T_1$.

The second case is when $n$ is such that $\fint \varphi(T^{S_n},x)dx \leq B(m_1, T)\left( \int \psi_0 \right)^2$.  This implies that $m \leq m_1$.  Indeed, using the definition of $S_n$, we obtain that 
\[
	B(m_1,T)\left( \int \psi_0 \right)^2
		\geq \fint \varphi(T^{S_n},x) dx
		= S_n
		= u_0(nL) \left( \int \psi_0 \right)^2
		\geq B(m,T) \left( \int \psi_0 \right)^2.
\]
Thus, $B(m,T) \leq B(m_1,T)$, implying that $m \leq m_1$.  Hence we may find $T_2$, depending only on $m_1$, $f$, and $\epsilon$, such that if $T \geq T_2$ then $T^{S_n}$ is close enough to $S(m,T)$ that
\[
	\|\varphi(T^{S_n}+T, \cdot) - \varphi(S(m,T)+T,\cdot)\|_{L^\infty} \leq \epsilon/8.
\]
In addition, increasing $T_2$ if necessary, we may apply \cref{p_convergence_smt} to get that if $T \geq T_2$ then
\[
	\|\varphi(S(m,T)+ T, \cdot) - \phi^{T,m}(T,\cdot)\|_\infty \leq \epsilon /16.	
\]
The combination of these two inequalities with \cref{e_u_phi_close} implies that if $T \geq T_1 + T_2$, then
\[
	\|\varphi(T^{S_n} + T, \cdot) - u(T,\cdot)\|_{L^\infty([nL,nL+L])} \leq \epsilon.
\]

Finally, we handle the case when neither of these is true.  That both of these conditions do not hold, along with parabolic regularity, implies that we may find a positive constant $\theta < 1$, which depends only on $m_1$ and $f$, such that
\[
	\theta \fint \varphi(T^{S_n}, x) dx \leq B(m_1,T) \left( \int \psi_0\right)^2.
\]
Using \cref{p_exponential} we may find $t_0>0$, depending only on $f$ and $m_1$, such that
\[
	\fint \varphi(T^{S_n}, x) dx \leq B(m_1,T-t_0) \left( \int \psi_0\right)^2.
\]
Following the work above, we then have that $B(m,T) \leq B(m_1,T-t_0)$.  Hence we obtain
\[
	(1- m) = \fint (1- \phi^{T,m}(T,x))dx
		\geq \fint (1-\phi^{T-t_0,m_1}(T, x))dx
		\geq C e^{f_1 t_0}
.\]
The constant $C$ depends only on $m_1$ and $f$, and the last inequality comes from considering a constant multiple of $\psi_1 e^{f_1 t}$ as a sub-solution to $1 - \phi^{T-t_0}$ on the interval $[T-t_0, T]$.  This shows that $m$ is bounded away from $1$ independent of $T$.  Hence we may argue exactly as in the last paragraph, finishing the proof.
\end{proof}

We will now show how to conclude \cref{p_convergence_phit}.  The main idea is largely the same: when $m$ is large enough, both $u$ and $\phi^{T,m}$ are near $1$, and when $m$ is bounded from $1$, we may take the limit in $T$ uniformly in $m$.
\begin{proof}[Proof of \cref{p_convergence_phit}]
Fix $\epsilon > 0$, and we will choose $T$, independent of $m$, such that
\[
	\|u(T) - \phi^{T,m}(T) \|_{L^\infty(I_{m,T})} \leq \epsilon.
\]
To this end, we utilize the work from \cref{s_inhomogeneous}.

First, we may use parabolic regularity to find $m_1<1$ depending only on $f$ and $\epsilon$ such that $1 - \phi^{T,m_1}(T,x) \leq \epsilon/2$.  We may now consider, separately, the cases when $m \leq m_1$ and when $m \geq m_1$.

If $m \geq m_1$, using a sub-solution, $v = \phi^{T, m_1} \underline{w}$, as in \cref{ss_lower_bound_inhomogeneous}, we may choose $T$, depending only on $\epsilon$ and $f$, such that
\[
	\underline{w}(T,x) \geq (1 - \epsilon)/(1 - \epsilon/2),
\]
for any $x \leq \max I_{m_1,T}$.  Hence, we have that
\[
	1 \geq u(T,x)
		\geq v(T,x)
		\geq \phi^{T, m_1}(T,x) w_\ell
		\geq (1 - \epsilon/2) \frac{1 - \epsilon}{1- \epsilon/2}
		= 1 - \epsilon.
\]
for any $x \leq \max I_{m_1,T}$.  By the maximum principle, we can see that $\phi^{T,m} \geq \phi^{T,m_1}$ if $m \geq m_1$.  Hence we have that, if $m \geq m_1$
\[
	\|u(T) - \phi^{T,m}\|_{L^\infty(I_{m,T})} \leq \epsilon.
\]

The case when $m \leq m_1$ follows from our work above.  Indeed, since $m$ is bounded away from $1$, we may choose $T$ such that $\overline{w}$ and $\underline{w}$ from \cref{ss_upper_bound_inhomogeneous} and \cref{ss_lower_bound_inhomogeneous}, respectively, satisfy
\[
	1 - \epsilon/2 \leq \underline{w}(T,x),
		~~\text{ and }~~
		\overline{w}(T,x) \leq 1 + \epsilon/2,
\]
for any $x\in I_{m,T}$.  Hence we have that, for any $x\in I_{m,T}$,
\[
	\left(1-\frac{\epsilon}{2}\right) \phi^{T,m}(T,x)
		\leq \underline{w}(T,x) \phi^{T,m}(T,x)
		\leq u(T,x)
		\leq \overline{w}(T,x) \phi^{T,m}(T,x)
		\leq \left(1 + \frac{\epsilon}{2}\right) \phi^{T,m}(T,x),
\]
which finishes the proof.
\end{proof}

\section{Well-posedness of the global in time solution}\label{s_wellposed}

In this section, we establish the existence and uniqueness (up to translation) of the solution to \cref{e_logistic_global}.  In other words, we prove the following proposition.
\begin{prop}
There exists a unique global in time solution of \cref{e_logistic_global}.
\end{prop}

We will prove this in a series of three lemmas.  First, we will exhibit the existence of a solution.  Afterwards, we show that any solution to \cref{e_logistic_global} must look grow and decay exponentially as $t$ tends to $-\infty$ and $\infty$, respectively.  Finally, we will use this qualitative property and the sliding method to establish uniqueness.

\begin{lem}\label{p_existence}
There exists a solution of \cref{e_logistic_global}.
\end{lem}

\begin{proof}
We establish existence by considering the limit of family of initial value problems.  Define $\phi^n$ to be the solution to the parabolic equation on $[T_n, \infty]\times \mathbb{T}$
\[\begin{cases}
	\phi^n_t = \phi^n_{xx} + f(x,\phi^n),\\
	\fint \phi^n(0,x) dx = 1/2,\\
	\phi^n(T_n,x) \equiv 1/n,
\end{cases}\]
where $T_n\leq 0$ is an unknown.  It is easy to check that $(n \min \psi_0)^{-1}\psi_0(x) e^{f_0(t-T_n)}$ is a super-solution, where $\psi_0$ and $f_0$ are the principle eigenfunction and eigenvalue of $\Delta + f_u(x,0)$.  Thus, we obtain
\[
	\frac{1}{2} = \fint \phi^n(0,x)dx \leq \frac{\|\psi_0\|_\infty}{n \min \psi_0} e^{-f_0 T_n}.
\]
This implies that $T_n \leq C \log(n)$, for some constant $C$.  Hence we may use parabolic regularity, along with the compactness of H\"older spaces to obtain convergence along a subsequence to a global in time function $\phi$ which satisfies
\[\begin{cases}
	\phi_t = \phi_{xx} + f(x,\phi),\\
	\fint \phi(0,x)dx = 1/2,
\end{cases}\]
finishing the proof.
\end{proof}


Now we establish that the asymptotics of any solution $\varphi$ to \cref{e_logistic_global} look like the solutions to the linearized problems around 0 and 1.  This amounts to proving \cref{p_exponential_global}, as stated in \cref{s_global_in_time}.

\begin{proof}[Proof of \cref{p_exponential_global}]
We will prove the claim for $t$ tending to $-\infty$.  The result for $t$ tending to infinity follows by looking at $\tilde\varphi(t,x) = 1 - \varphi(-t,x)$ and arguing similarly.  Our strategy is to first show that $\phi$ is bounded above and below by a multiple of $e^{f_0 t}$.  Then we will appeal to parabolic regularity to obtain the result.

To this end, we first assume, by contradiction, that $e^{-f_0 t} \varphi(t,x)$ tends to zero along some sequence $t_n$, tending to $-\infty$, for some choice of $x_n$.  By the parabolic Harnack inequality, it must converge to zero uniformly in space.  Hence, we may assume that, for any $\epsilon$, there exists $N$ such that if $n \geq N$ then $\|e^{-f_0 t_n} \varphi(t_n)\|_\infty \leq \epsilon$.  In this case, we may apply \cref{p_exponential} to find $t_n$ large enough that $\varphi^{-t_n, 1/4}(t-t_n,x)$ is a super-solution to $\varphi$ on $[-t_n,0]\times \T$.  Using this inequality at time $t=0$, we have that
\[
	\frac{1}{4}
		= \fint \varphi^{-t_n,1/4} (-t_n,x)dx
		\geq \fint \varphi(0,x)dx
		= \frac{1}{2}.
\]
This is clearly a contradiction, finishing the lower bound.

Now we will show that there is some $C$ such that $e^{-f_0 t} \varphi(t,x) \leq C$ for $t \leq 0$.  To this end, suppose that there is some sequence $t_n$ tending to $-\infty$ such that $e^{-f_0 t_n} \varphi(t_n,x_n) \to \infty$ for some sequence $x_n$.  Again, by the parabolic Harnack inequality, this implies that $\min e^{-f_0 t_n} \varphi(t_n) \to \infty$.  Hence, we may apply \cref{p_exponential} and take $n$ large enough that $\varphi^{-t_n,3/4}(t-t_n,x)$ is a sub-solution of $\varphi$ on $[-t_n,0]\times\T$.  Using this inequality at time $t = 0$ yields
\[
	\frac{3}{4}
		= \fint \varphi^{-t_n,3/4}(-t_n,x)dx
		\leq \fint \varphi(0,x)dx
		= \frac{1}{2}.
\]
This is a contradiction, giving us the desired upper bound.


Now, letting $\Phi(t,x) = e^{-f_0 t} \varphi(t,x)$, we will show that $\Phi$ converges to $\alpha \psi_0$.  To this end, we may write $\Phi_1 = \int \Phi(t,x) \psi_0(x)dx$ and $\Phi_2 = \Phi - \psi_0 \Phi_1$.  Then as in our work in \cref{p_linearized_error}, we may see that there are positive constants $C$ and $\theta$, independent of time such that
\[\begin{split}
	&0 \geq \dot{\Phi}_1 \geq - C e^{\delta f_0 t} \Phi_1,\\
	&\frac{d}{dt} \|\Phi_2\|_2^2 + 2 \theta \|\Phi_2\|_2^2 \leq C e^{\delta f_0 t} \|\Phi_2\|_2.
\end{split}\]
Solving these differential inequalities and using that $\Phi$ is bounded for all time gives us that, as $t$ tends to $-\infty$, $\Phi_1(t) \to \alpha$, for some $\alpha$, and that $\Phi_2(t) \to 0$.  Since $\Phi$ is bounded away from zero, we get that $\alpha >0$.  This gives us the desired result.
\end{proof}

Finally, we use this qualitative information, along with the sliding method, to prove uniqueness of any solution to \cref{e_logistic_global}.

\begin{lem}\label{p_uniqueness}
Any solution to \cref{e_logistic_global} is unique.
\end{lem}
\begin{proof}
Suppose that $\varphi^1$ and $\varphi^2$ are both solutions to \cref{e_logistic_global} with $\fint \phi^i(0,x) dx = 1/2$.  Using \cref{p_exponential_global}, we may choose $h$ large enough such that
\[
	\varphi^1(t,x) \leq \varphi^2(t+h,x),
\]
holds for all $t$ and all $x$.  Then we define
\begin{equation}\label{e_h_0}
	h_0 = \inf \{h : \varphi^1(t,x) < \varphi^2(t+h,x)\}.
\end{equation}
Since $\varphi^1$ and $\varphi^2$ have the same spatial average at $t = 0$, then $h_0 \geq 0$.  We wish to show that $h_0 = 0$.  To this end, let us assume that $h_0 > 0$.  There are two cases.  Using \cref{p_exponential}, we can find $\alpha_1$ and $\alpha_2$ such that $\varphi^1(t,x) \to \alpha_1 \psi_0$ as $t$ tends to $-\infty$  that $\varphi^2(t,x) \to \alpha_2 \psi_0$ as $t$ tends to $-\infty$ .  The first case is that $\alpha_1 = e^{f_0 h_0} \alpha_2$.  In other words, $\varphi^1$ and $\varphi^2$ touch at time $t=-\infty$.  Define $\eta(t,x) = \varphi^2(t+h_0,x) - \varphi^1(t,x)$ and $\eta$ satisfies
\[
	\eta_t = \eta_{xx} + \frac{f(x,\varphi^2(t+h_0,x)) - f(x,\varphi^1(t,x))}{\varphi^2(t+h_0,x) - \varphi^1(t+h,x)} \eta.
\]
Fix $T<0$ to be determined later.  Since $\phi^i \sim \alpha_i e^{f_0 T}$ and since $\alpha_1 = e^{f_0 h_0} \alpha_2$, we notice that $\eta(T,x) \leq o(1) e^{f_0 T}$ where $o(1)$ tends to zero as $T$ tends to $-\infty$.  From this it is easy to see that $o(1)\psi_0 e^{f_0 t}$ is a super-solution to equation for $\eta$ starting from any time $T$.  This implies that $\eta(0,x) \leq o(1)$.  Since this holds independent of our choice of $T$, we get that that $\eta(0,x) = 0$.  

On the other hand, $\varphi^2$ is bounded away from $0$ and $1$ on the time interval $[0,h_0]$.  Hence, there exists $\theta>0$ such that $f(x,\varphi^2) \geq \theta$ on the time interval $[0,h_0]$.  Integrating the equation for $\varphi^2$ in time and space, along with this fact and the fact that $\eta(0,x) = 0$, yields
\[\begin{split}
	0
		&= \fint \eta(0,x) dx
		= \fint \varphi^2(h_0,x) dx - \fint \varphi^1(0,x)dx
		= \fint \varphi^2(h_0,x) dx - m\\
		&= \fint \varphi^2(h_0,x)dx - \fint \varphi^2(0,x)dx
		= \int_0^{h_0} \fint f(x,\varphi^2(s,x)) dx ds
		\geq \int_0^{h_0} \theta ds = h_0 \theta > 0.
\end{split}\]
The second equality on the second line comes from integrating the equation for $\varphi^2$ in time and space.  The above inequality is clearly a contradiction.

Hence, it must be that, instead, $\alpha_1 < e^{f_0h_0} \alpha_2$.  Choose $\epsilon>0$ small enough that $\alpha_1 e^{\epsilon} < e^{h_0 - 2\epsilon} \alpha_2$.  Take $t_0$ to be a very large negative number such that if $t\leq t_0$ then
\[
	\varphi^1(t,x) \leq \alpha_1 e^\epsilon \psi_0(x) e^{f_0 t}
\]
and
\[
	\varphi^2(t + h_0) \geq \alpha_2 e^{f_0(t + h_0 - \epsilon)} \psi_0(x).
\]
Notice then that $\varphi^2(t+h_0 - \epsilon, x) > \varphi^1(t,x)$ for all $t < t_0$.  In addition, the comparison principle assures us that $\varphi^2(t+h_0 - \epsilon,x) > \varphi^1(t,x)$ for all $t \geq t_0$ as well.  This contradicts our choice of $h_0$ in \cref{e_h_0}.

Hence we may conclude that $h_0$ must be zero.  This implies that $\varphi^1(t,x) \leq \varphi^2(t,x)$ for all $t$ and all $x$.  The proof may be applied just as easily to show that $\varphi^2 \leq \varphi^1$ as well.  Hence we conclude that $\varphi^1 \equiv \varphi^2$, concluding the proof.
\end{proof}

\section*{Acknowledgements}
The author wishes to thank Jun Gao, Yu Gu, and Lenya Ryzhik for many helpful discussions.

\bibliography{fast_prop.bib}{}
\bibliographystyle{amsplain}

\end{document}